\documentclass[11pt]{article}

\usepackage{pgf,tikz}
\usetikzlibrary{arrows.meta, positioning, shapes.geometric}

\usepackage{mathrsfs}
\usetikzlibrary{arrows}

\usepackage{booktabs}

\usepackage{empheq}

\usepackage{pgfplots}
\pgfplotsset{compat=newest}

\usepackage{graphicx} 
\graphicspath{{figures/}} 

\usepackage{fullpage, url,amsmath,amsfonts,amssymb,mathtools,mathrsfs,graphicx,  algorithm, float, sansmath,epstopdf,color,caption,enumitem,tabularx}
\usepackage[final]{pdfpages}
\usepackage{amsthm}
\usetikzlibrary{automata,topaths}
\usetikzlibrary{decorations.pathreplacing,shapes.misc}
\usepackage{fancyhdr}

\def\beq{\begin{equation}}
\def\eeq{\end{equation}}
\def\baq{\begin{eqnarray}}
\def\eaq{\end{eqnarray}}
\def\baqn{\begin{eqnarray*}}
\def\eaqn{\end{eqnarray*}}

\usepackage{multirow}
\usetikzlibrary{calc,arrows}
\theoremstyle{plain}

\newtheorem{remark}{Remark}

\newtheorem{example}{Example}
\newtheorem{theorem}{Theorem}

\newtheorem{corollary}[theorem]{Corollary}
\newtheorem{proposition}[theorem]{Proposition}

\usepackage{blindtext}

\usepackage[colorlinks,linkcolor=blue,citecolor=red]{hyperref}

\usepackage{xcolor, framed}

\newcommand{\R}{{\mathbb R}}

\newcommand{\interior}{{\rm int}\kern 0.06em}

\def\<{\langle}
\def\>{\rangle}

\usepackage{ulem}

\usepackage{subcaption}

\usepackage{pict2e}

\if
{

\usepackage[pageref]{backref}
\renewcommand*{\backrefalt}[4]{%
\ifcase #1 %
(Not cited)%
\or
(Cited on p.~#2)%
\else
(Cited on pp.~#2)%
\fi
}

}
\fi

\begin{document}
\title{{New  General Fixed-Point Approach to Compute the Resolvent of Composite Operators}}
\author{Samir Adly\thanks{Laboratoire XLIM, Universit\'e de Limoges,
123 Avenue Albert Thomas,
87060 Limoges CEDEX, France\vskip 0mm
Email: \texttt{samir.adly@unilim.fr}}
 \qquad Ba Khiet Le \thanks{Optimization Research Group, Faculty of Mathematics and Statistics, Ton Duc Thang University, Ho Chi Minh City, Vietnam\vskip 0mm
 E-mail: \texttt{lebakhiet@tdtu.edu.vn}}
 }
\date{\today}
\maketitle
\begin{abstract}
In this paper,  we  propose a new general  and stable  fixed-point approach to  compute the resolvents of the composition of a set-valued maximal monotone operator with a linear bounded mapping. Weak, strong and linear convergence of the proposed algorithms are obtained. Advantages of our method over the existing approaches are also thoroughly analyzed. 
\end{abstract}
\section{Introduction}
Resolvents of maximal monotone operators play an essential role in many fundamental algorithms in optimizations such as proximal point algorithm, Douglas-Rachford splitting, forward-backward splitting and their generalized forms (see, e.g., \cite{Attouch0, Attouch1,Bauschke,Tseng} and the references therein).  
Our aim is to find a new efficient way to compute the resolvent of the composite operator $C^T\mathcal{M}C$,  where $\mathcal{M}: H_2 \rightrightarrows H_2$ is  maximal monotone and $C: H_1\to H_2$ is a linear bounded mapping with its adjoint $C^T: H_2 \to H_1$ where $H_1, H_2$ are  real Hilbert spaces. 
 {
  It can be rigorously proved that the composed operator \( C^T\mathcal{M}C : H_1 \to H_1 \), defined by \( C^T\mathcal{M}C(x) := \{C^Tu^* | u^* \in \mathcal{M}(Cx)\} \), exhibits monotonicity. Such operators are not only frequently encountered in the study of elliptic partial differential equations but also play a significant role in optimization problems, Lur'e dynamical systems, signal processing, and machine learning, particularly in the context of large-scale data analysis (see, e.g., \cite{ahl2,Brezis,br0,BT,chen,Micchelli,Micchelli1,Moudafi,L1} and references therein). Moreover, they encapsulate, as a particular instance, the scenario where an operator is the pointwise sum of two or more individual operators. However, without additional constraints, the operator \( C^T\mathcal{M}C \) may not necessarily satisfy the property of maximal monotonicity. For sufficient conditions that guarantee maximal monotonicity, we refer to the works cited in \cite{Zalinescu}.
Further, let \( f : H_2 \to \mathbb{R}\cup\{+\infty\} \) be a function that is convex, proper, and lower semicontinuous.  The composition \( f \circ C \) retains the properties of convexity and lower semicontinuity. Invoking the chain rule of convex analysis yields the relation:
\[
C^T\:(\partial f)\:C\subseteq \partial (f \circ C).
\]
This inclusion becomes an equality under the constraint qualification \( 0 \in \text{int}(\text{rge}\, (C) - \text{dom} f) \), implying that
\[
\partial (f \circ C) = C^T\:(\partial f)\:C.
\]
From a numerical perspective, computing the resolvent of \( \partial (f \circ C) \) is crucial for the implementation of algorithms such as the proximal point algorithm. Accurate computation of this resolvent can significantly enhance the performance and convergence properties of  algorithms in solving optimization problems.}

Mathematically, for given $\lambda >0, y\in H_1$ we want to compute the resolvent $J_{\lambda C^T\mathcal{M}C}(y):=x\in H_1$ satisfying 
\beq\label{comp}
y\in x+ \lambda C^T\mathcal{M}C x.
\eeq
Note that  (\ref{comp}) can be rewritten as follows 
\begin{equation}\label{lmiso}
\left\{
\begin{array}{l}
x=y-\lambda C^T v\\ \\
v\in \mathcal{M}C x,
\end{array}\right.
\end{equation}
and  thus  
$$
v \in \mathcal{M}(Cy-\lambda CC^Tv) \Leftrightarrow v \in (\mathcal{M}^{-1}+\lambda CC^T)^{-1}(Cy).
$$
The equality can be obtained in the latter inclusion if $E:=CC^T$ is invertible \cite{Robinson} since then the operator $\mathcal{M}^{-1}+\lambda CC^T$ is strongly monotone. However the operator $\mathcal{M}^{-1}$  makes the computation complex. 

On the other hand, if $E$ is invertible, similarly to the classical equality $(\mathcal{M}^{-1}+\lambda I)^{-1}=\frac{1}{\lambda}(I-J_{\lambda \mathcal{M}})$ one can show that (see Proposition \ref{formula})
$$
(\mathcal{M}^{-1}+\lambda E)^{-1}=\frac{1}{\lambda}E^{-1}(I-J_{\lambda E\mathcal{M}})
$$
where $J_{\lambda E\mathcal{M}}$ is the resolvent of $\lambda E\mathcal{M}$ defined by 
$$
J_{\lambda E\mathcal{M}}:=(I+\lambda E\mathcal{M})^{-1},
$$
and $E\mathcal{M}$ may not be maximal monotone. 
 It gives a new explanation to the algorithm proposed by  Fukushima \cite{Fukushima}, which does not require $\mathcal{M}^{-1}$. {{Nevertheless} even in a finite-dimensional setting, calculating $E^{-1}$ remains costly when the size of matrix $C$ is too large.}

\vskip 2mm
In \cite{Micchelli}, Micchelli, Chen, and Xu introduced an new efficient approach using fixed-point techniques to address certain limitations in the case where $\mathcal{M} = \partial f$. This methodology was further explored by A. Moudafi in \cite{Moudafi} for a general maximal monotone operator $\mathcal{M}$. They established that
\begin{equation}\label{commou}
J_{C^T \mathcal{M} C} y = y - \mu C^T u,
\end{equation}
where $u$ is the fixed point of the operator $\mathcal{Q}$, defined by 
\begin{equation}\label{commou1}
\mathcal{Q} := (I - J_{\mathcal{\frac{1}{\mu} M}}) \circ F \mbox{ with } F(u) := Cy + (I - \mu CC^T)u,
\end{equation}
 for some $\mu > 0$. The fixed point $u$ can be determined using well-known algorithms such as the Krasnoselskii--Mann algorithm to obtain weak convergence, provided that 
\beq\label{fc}
\| I - \mu CC^T \| \leq 1.
\eeq
 To compute $J_{\lambda C^T \mathcal{M} C}$, equation (\ref{commou}) can be applied by setting $\mathcal{M}:= \lambda \mathcal{M}$ and the convergence condition  still involves only the parameter $\mu$ while the parameter $\lambda$ is not used. {It is crucial to note that the substitution $C' = \sqrt{\lambda} C$ is not feasible because $\mathcal{M}$ is generally nonlinear}. Thus this approach is limited by the necessity of a very small $\mu$ if $\Vert C\Vert $ is large. The method was extended to compute $J_{\mathcal{P} + C^T \mathcal{M} C}$ in \cite{chen}, where $\mathcal{P}$ is also a maximal monotone operator, with the explicit condition $\mu \in (0, 2/\| C \|^{2})$ for strong convergence using averaged operators. It is demonstrated (Lemma \ref{lem1}) that the condition $\| I - \mu CC^T \| \leq 1$ is weaker than the explicit condition $\mu \in (0, 2/\| C \|^{2})$. However the explicit condition can be easily used  in practice. We prove that the conditions $\mu \in [0, 2/\| C \|^{2}]$ and  $\| I - \mu CC^T \| \leq 1$ are equivalent if $C$ is a matrix (Proposition \ref{lem1}). 

Motivated by these considerations, we propose a new direct fixed-point approach to compute $J_{\lambda C^T \mathcal{M} C}$ for any $\lambda > 0$, based on the identity $(\mathcal{M}^{-1} + \alpha I)^{-1} = \frac{1}{\alpha}(I - J_{\alpha \mathcal{M}})$. We show that (Corollary \ref{algo2n}) 
\begin{equation}\label{oural}
J_{\lambda C^T \mathcal{M} C} y = y - \lambda \mu C^T u,
\end{equation}
where $u$ is the fixed point of the operator 
\beq\label{oural1}
\mathcal{Q}(u) :=  (I - J_{\frac{1}{\mu} \mathcal{M}}) \Big(Cy + ( I - \lambda \mu CC^T) u\Big).
\eeq
 {Our associated algorithm is strongly convergent under the condition 
 \beq
 \| I - {\lambda}{\mu} CC^T \| \leq 1.
 \eeq} 
 This condition is flexible since it uses both parameters $\lambda$ and $\mu$ especially beneficial when $\Vert C \Vert$ is large. If $\lambda=1$, our fixed-point equation  (\ref{oural}) - (\ref{oural1}) becomes the  fixed-point algorithm (\ref{commou}) - (\ref{commou1}). Furthermore, our algorithm is showed to be more stable (Example \ref{ex1}).
 In addition, if $CC^T$ is positive definite, $\lambda$ and $\mu$ can be selected to achieve linear convergence.  We also provide an application of our results to compute  equilibria of set-valued Lur'e dynamical systems (Example \ref{ex2}). {The technique can be applied similarly to compute the resolvent $\mathcal{M}_1+C^T\mathcal{M}_2C$ where $\mathcal{M}_1:  H_1 \rightrightarrows H_1, \mathcal{M}_2:  H_2 \rightrightarrows H_2$ are maximal monotone operators. }

The paper is organized as follows.  In Section \ref{sec2}  we recall some definitions and useful results in the theory of monotone operators.  In Section \ref{sec3}, we provide a new way to compute the resolvent of the composite operators efficiently. The paper  ends in Section \ref{sec5} with some conclusions.

\section{Notations and preliminaries} \label{sec2}
Let  be  given  a {real} Hilbert space  $H$ with {the inner product $\langle \cdot,\cdot \rangle$ and the associated norm $\Vert \cdot \Vert$.}  
\noindent The domain, the range, the graph and the inverse of a set-valued mapping $\mathcal{M}: {H}\rightrightarrows {H}$ are defined respectively by
$${\rm dom}(\mathcal{M})=\{x\in {H}:\;\mathcal{M}(x)\neq \emptyset\},\;\;{\rm rge}(\mathcal{M})=\displaystyle\bigcup_{x\in{H}}\mathcal{M}(x)\;\;$$
 and
 $$\;\;{\rm gph}(\mathcal{M})=\{(x,y): x\in{H}, y\in \mathcal{M}(x)\},\;\;\;\mathcal{M}^{-1}(y)=\{ x\in H: y\in \mathcal{M}(x) \}.$$

\noindent {The mapping $\mathcal{M}$ is called monotone if 
$$
\langle x^*-y^*,x-y \rangle \ge 0, \;\;\forall \;x, y\in H, x^*\in \mathcal{M}(x) \;{\rm and}\;y^*\in \mathcal{M}(y).
$$
Furthermore, if there is no monotone operator $\mathcal{N}$ such that the graph of $\mathcal{M}$ is strictly {included} in  the graph of $\mathcal{N}$, then $\mathcal{M}$ is called maximal monotone. The subdifferential of a proper lower semicontinuous convex function is an important example of maximal monotone mappings.
 \noindent {The resolvent  of $\mathcal{M}$ is defined  as follows
$$
J_\mathcal{M}:=(I+\mathcal{M})^{-1}.
$$
}
It is known that (see, e.g, \cite{Rockafellar}) $J_\mathcal{M}$ is firmly-nonexpansive, i.e., 
$$
\langle J_\mathcal{M} x-J_\mathcal{M}y, x-y \rangle \ge \Vert  J_\mathcal{M} x-J_\mathcal{M}y \Vert^2.
$$
{We summarize several classical properties of maximal monotone operators in the proposition below (we refer to \cite{AC} for example).}
{
\begin{proposition}{\rm (\cite{AC})}\label{Yosida}
Let $\mathcal{M}:  {H} \rightrightarrows  {H}$ be a maximal monotone operator and let $\lambda>0$. Then
\begin{enumerate}
  \item[{\rm (i)}]  The {resolvent} $J_{\lambda \mathcal{M}}:=(I+\lambda \mathcal{M})^{-1} $ is a {nonexpansive}  and single-valued map from $ {H}$ to $ {H}$.
\item[{\rm (i)}]  The {Yosida approximation} of $\mathcal{M}$  (of index $\lambda$) defined by $$\mathcal{M}_\lambda:=\frac{1}{\lambda}(I-J_{\lambda \mathcal{M}})=(\lambda I+\mathcal{M}^{-1})^{-1}$$ satisfies\\
\begin{enumerate}
  \item[{\rm (a)}]  for all $x\in \mathcal{H}$, $\mathcal{M}_\lambda(x)\in \mathcal{M}(J_{\lambda \mathcal{M}} (x))$;
\item[{\rm (b)}]  $\mathcal{M}_\lambda$ is $\frac{1}{\lambda}$-Lipschitz continuous   and also maximal monotone. 
\end{enumerate} 
\end{enumerate}
\end{proposition}
}
\noindent  A linear bounded mapping $D: H\to H$  is called
 \begin{itemize}
\item  {{\textit{positive semidefinite}, denoted}  $D\succeq 0$}, if 
$$\langle Dx,x \rangle \ge 0,\;\;\forall\;x\in H;$$
\item {{\textit{positive definite},  denoted}  $D\succ 0$}, if there exists $c>0$ such that 
$$
\langle Dx,x \rangle \ge c\|x\|^2,\;\;\forall\;x\in H.
$$
\end{itemize}
\begin{proposition}
Let {$E\succ 0$}, $\lambda >0$ and $\mathcal{M}:  {H} \rightrightarrows  {H}$ be a maximal monotone operator. Then  $J_{\lambda E\mathcal{M}}:=(I+\lambda E\mathcal{M})^{-1}$ is single-valued and Lipschitz continuous.
\end{proposition}
\begin{proof}
Let $y_i\in J_{\lambda E\mathcal{M}}x_i, i=1, 2.$ Then we have 
$$
x_i\in y_i+\lambda E\mathcal{M} (y_i) \Leftrightarrow \frac{1}{\lambda}E^{-1}(x_i-y_i)\in \mathcal{M}(y_i).
$$
The {operator} $E^{-1}$ is positive definite since $E$ is positive definite.  Using the monotonicity of $\mathcal{M}$, we have 
$$
\langle E^{-1}(x_1-x_2-y_1+y_2), y_1-y_2 \rangle \ge 0.
$$
Thus 
$$
c\Vert y_1-y_2\Vert^2\le \langle E^{-1}(y_1-y_2), y_1-y_2\rangle\le \langle E^{-1}(x_1-x_2), y_1-y_2 \rangle \le \Vert  E^{-1}\Vert \Vert x_1-x_2\Vert \Vert y_1-y_2\Vert,
$$
for some $c>0$ and the conclusion follows. 
\end{proof}
{The subsequent result offers a characterization of  $(\mathcal{M}^{-1} + \lambda E)^{-1}$, analogous to the classical identity $(\mathcal{M}^{-1} + \lambda I)^{-1} = \frac{1}{\lambda}(I - J_{\lambda \mathcal{M}})$ presented in Proposition \ref{Yosida}.}
\begin{proposition}\label{formula}
If {$E\succ0, \lambda>0$} and $\mathcal{M}:  {H} \rightrightarrows  {H}$ is a maximal monotone operator,  then 
$$
(\mathcal{M}^{-1}+\lambda E)^{-1}=\frac{1}{\lambda}E^{-1}(I-J_{\lambda E\mathcal{M}}).
$$
\end{proposition}
\begin{proof}
{It is easy to see that both operators in the last equality are single-valued}. Let $x\in H$ and $y:=\frac{1}{\lambda}E^{-1}(I-J_{\lambda E\mathcal{M}})(x)$. 
{We have,
\baqn
(\mathcal{M}^{-1}+\lambda E)^{-1}(x)=y &\Leftrightarrow& x\in (\mathcal{M}^{-1}+\lambda E)(y).\\
&\Leftrightarrow&x\in (\mathcal{M}^{-1}+\lambda E)(\frac{1}{\lambda}E^{-1}-\frac{1}{\lambda}E^{-1}J_{\lambda E\mathcal{M}})(x)\\
&\Leftrightarrow& x\in x-J_{\lambda E\mathcal{M}}x+\mathcal{M}^{-1}(\frac{1}{\lambda}E^{-1}-\frac{1}{\lambda}E^{-1}J_{\lambda E\mathcal{M}})(x) \\
&\Leftrightarrow& E^{-1}x \in E^{-1}J_{\lambda E\mathcal{M}}x+ \lambda \mathcal{M} (J_{\lambda E\mathcal{M}}x) \\
&\Leftrightarrow& x\in J_{\lambda E\mathcal{M}}x+\lambda E\mathcal{M} (J_{\lambda E\mathcal{M}}x).
\eaqn
The last inclusion is valid as it directly follows from the definition of $J_{\lambda E\mathcal{M}}$.
}
\end{proof}
\section{Main results} \label{sec3}
In this section, we establish a new general fixed-point equation  to compute efficiently the resolvents of the composite operators  $C^T\mathcal{M}C$ and  $\mathcal{M}_1+C^T\mathcal{M}_2C$ where $C: H_1\to H_2$ is a linear bounded mapping, $\mathcal{M}, \mathcal{M}_2: H_2 \rightrightarrows H_2, \mathcal{M}_1:  H_1 \rightrightarrows H_1$ are maximal monotone operators and $H_1, H_2$ are real Hilbert spaces. The key tool is the   Yosida approximation $\mathcal{M}_\alpha:=\frac{1}{\alpha}(I-J_{\alpha \mathcal{M}})=(\mathcal{M}^{-1}+\alpha I)^{-1}$, $\alpha>0$. 
\subsection{Resolvent of $C^T\mathcal{M}C$}
First we  show that the explicit condition $\alpha\in [0,2/\Vert C \Vert^2]$ is stronger than the condition $\Vert I-\alpha CC^T \Vert \le 1$. However the explicit condition can be easily used in practice. These conditions are equivalent if $C$ is a matrix.
%
%
{
\begin{proposition}\label{lem1}
Let $C: H_1 \to H_2$ be a linear bounded mapping between Hilbert spaces. If $\alpha \in \left[0, \frac{2}{\|C\|^2}\right]$, then $\|I - \alpha CC^T\| \leq 1$. In addition, if $H_1=\R^n$ and $H_2=\R^m$ (i.e., $C$ can be represented as a matrix), then the two conditions are equivalent.
\end{proposition}
\begin{proof}
First, suppose that $\alpha \in [0,2/\|C\|^2]$. For all $x\in H_1$, we have 
\begin{align*}
\|x-\alpha CC^T x\|^2 &= \|x\|^2 - 2\alpha \langle x, CC^T x \rangle + \alpha^2 \|CC^T x\|^2 \\
&\leq \|x\|^2 - 2\alpha \|C^T x\|^2 + \alpha^2 \|C\|^2\|C^T x\|^2 \\
&= \|x\|^2 - \alpha \|C^T x\|^2 (2 - \alpha \|C\|^2) \\
&\leq \|x\|^2
\end{align*}
where the last inequality follows from $\alpha \leq 2/\|C\|^2$. Thus, $\|I - \alpha CC^T\| \leq 1$.\\
Next, suppose that $H_1=\R^n$ and $H_2=\R^m$ and $\|I - \alpha CC^T\| \leq 1$. We will prove that $\alpha \in [0,2/\|C\|^2]$ by contradiction. Assume that $\alpha > 2/\|C\|^2$. Let $B = CC^T$; then $B$ is a symmetric, positive semi-definite matrix and $\|B\| = \|C\|^2$. By the spectral theorem, there exists a unit vector $x^* \in H_1$ such that $Bx^* = \|B\| x^*$. Then 
\[
(I - \alpha B)x^* = (1 - \alpha \|B\|)x^*
\]
where $1 - \alpha \|B\| < 1 - \frac{2}{\|B\|}\|B\| = -1$. Consequently,
\[
\|I - \alpha CC^T\| = \|I - \alpha B\| \geq \|(I - \alpha B)x^*\| = |1 - \alpha \|B\|| > 1.
\]
This contradicts our assumption that $\|I - \alpha CC^T\| \leq 1$. Therefore, we must have $\alpha \in [0,2/\|C\|^2]$.
\end{proof}
}

{{
\begin{theorem}\label{tmf}
Let $\lambda > 0$ and $y \in H_1$ be given. Then, the operator $J_{\lambda C^T \mathcal{M} C} y$ can be expressed as $y - \lambda C^T v$, where $v$ is the fixed point of the operator $\mathcal{N}:H_2\to H_2$ defined by
$$v\mapsto\mathcal{N}(v) := \mathcal{M}_{\frac{1}{\mu}} \Big(Cy + (\frac{1}{\mu} I - \lambda CC^T)v\Big),$$
for any $\mu > 0$ and $\mathcal{M}_{\frac{1}{\mu}}$ denotes the Yosida approximation of $\mathcal{M}$ with index $\frac{1}{\mu}$.\\
 Furthermore, if ${\lambda}{\mu} \leq \frac{2}{\|C\|^2}$, then $\left\|I - {\lambda}{\mu} CC^T\right\| \leq 1$, thereby ensuring that $\mathcal{N}$ is {nonexpansive}.
\end{theorem}
}}
\begin{proof} 
{From (\ref{lmiso}), we deduce that $v \in \mathcal{M}(Cy - \lambda CC^T v)$. We have
\begin{align*}
v \in \mathcal{M}(Cy - \lambda CC^T v) 
&\Leftrightarrow Cy - \lambda CC^T v \in \mathcal{M}^{-1}(v) \\
&\Leftrightarrow Cy + (\frac{1}{\mu} I - \lambda CC^T)v \in (\mathcal{M}^{-1} + \frac{1}{\mu} I)v \\
&\Leftrightarrow v = (\mathcal{M}^{-1} + \frac{1}{\mu} I)^{-1}(Cy + (\frac{1}{\mu} I - \lambda CC^T)v) \\
&\Leftrightarrow v = \mathcal{M}_{\frac{1}{\mu}} (Cy + (\frac{1}{\mu} I - \lambda CC^T)v),
\end{align*}
where  $\mathcal{M}_{\frac{1}{\mu}}=(\mathcal{M}^{-1}+\frac{1}{\mu} I)^{-1}$ (see Proposition \ref{Yosida}).} \\
If ${\lambda}{\mu}\le 2/\Vert C\Vert^2$, then  $\Vert I-{\lambda}{\mu} CC^T \Vert \le 1$ (Lemma \ref{lem1}) and the mapping $L(u):=Cy+(\frac{1}{\mu} I-\lambda CC^T)u$ is $\frac{1}{\mu}$-Lipschitz continuous. Since $\mathcal{M}_{\frac{1}{\mu}}$ is ${\mu}$-Lipschitz continuous, $\mathcal{N}$ is {nonexpansive}. 
\end{proof}
Theorem \ref{tmf} can be used to design a numerical algorithm to compute a fixed point of the operator $\mathcal{N}$ and hence the resolvent  $J_{\lambda C^T\mathcal{M}C}$.
{
\begin{theorem}\label{algokm}
Let $\lambda > 0$ and $y \in H_1$ be given. Choose $\mu>0$ such that ${\lambda}{\mu} \in (0, 2/\|C\|^2)$, and let $(\alpha_k) \subset (0,1)$ satisfy $\sum_{k=1}^\infty \alpha_k (1 - \alpha_k) = \infty$. We construct the sequence $(v_k)$ as follows:
\begin{equation*}
\textbf{Algorithm 1:} \quad v_0 \in H, \quad v_{k+1} = (1 - \alpha_k)v_k + \alpha_k \mathcal{N}(v_k), \quad k = 0, 1, 2, \ldots.
\end{equation*}
Then, the sequence $(v_k)$ converges weakly to a fixed point $v$ of $\mathcal{N}$, and $J_{\lambda C^T \mathcal{M} C} y = y - \lambda C^T v$.\\
Furthermore, if $\inf \alpha_k > 0$, then the sequence $(x_k)$ defined by $x_k := y - \lambda C^T v_k$ converges strongly to $J_{\lambda C^T \mathcal{M} C} y$.
\end{theorem}
}
\begin{proof}
Using Theorem  \ref{tmf} and Krasnoselskii-Mann algorithm (see, e.g., \cite{Bauschke}), the weak convergence of ($u_k$) to some  fixed point $u$ of $\mathcal{N}$ follows. It remains to prove the strong convergence of $(v_k)$. We have 
\beq\label{udecre}
\Vert v_{k+1}-v\Vert\le(1-\alpha_k) \Vert v_{k}-v\Vert+\alpha_k\Vert \mathcal{N}(v_k)-\mathcal{N}(v)\Vert\le \Vert v_{k}-v\Vert,
\eeq
 where we use the fact that  $v=\mathcal{N}(v)$. Thus the sequence $( \Vert v_{k}-v\Vert)$ is decreasing and converges. On the other hand,  (\ref{udecre}) can be rewritten as follows
 \baqn
 \Vert v_{k+1}-v\Vert 
 & \le&  \Vert v_{k}-v\Vert+\alpha_k(\Vert \mathcal{N}(v_k)-\mathcal{N}(v)\Vert- \Vert v_{k}-v\Vert)\\
 &\le&  \Vert v_{k}-v\Vert.
 \eaqn
{Let $k\to \infty$, we obtain 
\baq
\alpha_k(\Vert \mathcal{N}(v_k)-\mathcal{N}(v)\Vert- \Vert v_{k}-v\Vert)\to 0.
\eaq
Since $\inf \alpha_k =\alpha> 0$ and $\Vert \mathcal{N}(v_k)-\mathcal{N}(v)\Vert\le \Vert v_{k}-v\Vert$, we imply that
\baq
0=\lim_{k\to \infty}\alpha_k( \Vert v_{k}-v\Vert-\Vert \mathcal{N}(v_k)-\mathcal{N}(v)\Vert)\ge \lim_{k\to \infty}\alpha( \Vert v_{k}-v\Vert-\Vert \mathcal{N}(v_k)-\mathcal{N}(v)\Vert)\ge 0,
\eaq
which deduces that $ \lim_{k\to \infty}\Vert \mathcal{N}(v_k)-\mathcal{N}(v)\Vert= \lim_{k\to \infty} \Vert v_{k}-v\Vert.$
Note that
$$
\Vert CC^T(v_k-v)\Vert \le \Vert C \Vert \Vert C^T(v_k-v)\Vert. 
$$
Thus 
\baqn
 \Vert \mathcal{N}(v_k)-\mathcal{N}(v)\Vert^2&\le& {\mu^2}\Vert (\frac{1}{\mu} I-\lambda CC^T)(v_k-v)\Vert^2\\
  &\le& \Vert v_{k}-v\Vert^2- 2 {\lambda}{\mu} \Vert C^T(v_k-v)\Vert^2 + {\lambda^2}{\mu^2}\Vert CC^T(v_k-v)\Vert^2\\
   &\le& \Vert v_{k}-v\Vert^2-  2 \frac{\lambda\mu}{ \Vert C \Vert^2}\Vert CC^T(v_k-v)\Vert^2+ {\lambda^2}{\mu^2}\Vert CC^T(v_k-v)\Vert^2\\
 &\le& \Vert v_{k}-v\Vert^2-{\lambda^2}{\mu^2}(\frac{2}{\lambda\mu\Vert C\Vert^2}-1)\Vert CC^T(v_k-v)\Vert^2.
\eaqn
Let $k\to \infty$, since ${\lambda}{\mu} \in (0, 2/\|C\|^2)$ we obtain that $\Vert CC^T(v_k-v)\Vert\to 0$. \\
Note that $(v_k)$ is bounded, we infer that $\Vert C^T(v_k-v)\Vert^2=\langle CC^T(v_k-v), v_k-v \rangle\le   \Vert CC^T(v_k-v)\Vert \Vert  v_k-v \Vert \to 0$, {which means that $C^Tv_k\to C^Tv$ strongly}. Therefore $x_k= y - \lambda C^T v_k$ converges strongly to $J_{\lambda C^T\mathcal{M}C}y=y-\lambda C^T v$.}
\end{proof}

Using Theorems \ref{tmf} and \ref{algokm},  by letting $v=\mu u$ and the fact that $\mathcal{M}_{\frac{1}{\mu}}=\mu  (I - J_{\frac{1}{\mu} \mathcal{M}})$, we obtain the following results.

\begin{corollary}\label{algo2n}
Let $\lambda > 0$ and $y \in H_1$ be given. Then, the operator $J_{\lambda C^T \mathcal{M} C} y=y - \lambda \mu C^T u$, where $u$ is the fixed point of the operator $\mathcal{Q}:H_2\to H_2$ defined by
$$u\mapsto\mathcal{Q}(u) :=  (I - J_{\frac{1}{\mu} \mathcal{M}}) \Big(Cy + ( I - \lambda \mu CC^T) u\Big), \mu>0.$$
 Furthermore, if ${\lambda}{\mu} \leq \frac{2}{\|C\|^2}$, then $\left\|I - {\lambda}{\mu} CC^T\right\| \leq 1$, thereby ensuring that $\mathcal{Q}$ is {nonexpansive}.
\end{corollary}

\begin{corollary}\label{alkm}
Let $\lambda > 0$ and $y \in H_1$ be given. Choose $\mu>0$ such that ${\lambda}{\mu} \in (0, 2/\|C\|^2)$, and let $(\alpha_k) \subset (0,1)$ satisfy $\sum_{k=1}^\infty \alpha_k (1 - \alpha_k) = \infty$. We construct the sequence $(u_k)$ as follows:
\begin{equation*}
\textbf{Algorithm 2:} \quad u_0 \in H_2, \quad u_{k+1} = (1 - \alpha_k)u_k + \alpha_k \mathcal{Q}(u_k), \quad k = 0, 1, 2, \ldots.
\end{equation*}
Then, the sequence $(u_k)$ converges weakly to a fixed point $u$ of $\mathcal{Q}$, and $J_{\lambda C^T \mathcal{M} C} y = y - \lambda \mu C^T u$.\\
Furthermore, if $\inf \alpha_k > 0$, then the sequence $(x_k)$ defined by $x_k := y - \lambda \mu C^T u_k$ converges strongly to $J_{\lambda C^T \mathcal{M} C} y$.
\end{corollary}
{
\begin{remark}\normalfont
The algorithm developed by Micchelli-Chen-Xu \cite{Micchelli} and later by Moudafi \cite{Moudafi} is formulated using only a single parameter $\mu$ within the convergence condition. The convergence is weak under the condition $\| I - \mu CC^T \| \leq 1$.  However, when $\|C\|$ is large, this can be computationally infeasible. In contrast, Theorem \ref{algokm} allows for the direct computation of $J_{\lambda C^T\mathcal{M}C}$ with strong convergence under the condition  $\| I - \lambda \mu CC^T \| \leq 1$. This condition, which involves both parameters $\mu$ and $\lambda$, offers greater flexibility in their selection, enabling them to be neither excessively large nor small, especially when $\|C\|$ is substantial. It is also noteworthy that if $\lambda=1$, Algorithm  2 becomes the algorithm used in \cite{Micchelli,Moudafi}, which is an important algorithm in computing the resolvent of the composite operators. Note that even $C$ is not big, our algorithms are also more stable as showed in the following example. 
\end{remark}
}
\begin{example} \label{ex1}
 Let us provide a simple example to show the advantage of our extension over the  Micchelli-Chen-Xu's approach. Let $\mathcal{M}=\partial \Vert \cdot \Vert_1$ and
$$
C=\begin{bmatrix}
1\;\;\; 3\;\;\; 7\;\; \;0\;\; \;8 \\
2 \;\;\;4 \;\;\;5 \;\;\;8\; \;\;7 \\
7\;\;\; 9\; \;\;6\; \;\;0 \;\;\;1 \\
2\;\;\; 0\;\; \;1\;\; \;4 \;\;\;7 \\
2 \;\;\;5\; \;\;8\; \;\;3\; \;\;8 \\
\end{bmatrix}.
$$
Then 
$$
CC^T=\begin{bmatrix}
123& 105 &84&65&137 \\
105  &158 &87  & 90 & 144 \\
 84  & 87 & 167  &27  & 115 \\
 65&  90  & 27&  70&   80 \\
137&144& 115&   80 & 166
\end{bmatrix}.
$$
and $\Vert CC^T\Vert=532.64.$
It is known that (see, e.g., \cite{Micchelli})
$$
J_{{\frac{1}{\mu}} \mathcal{M}}(x)=\Big(\max\big(\vert x_1 \vert-\frac{1}{\mu}, 0\big){ \rm sign}(x_1), \ldots, \max\big(\vert x_5 \vert-\frac{1}{\mu}, 0\big){ \rm sign}(x_5)\Big),
$$
where $x=(x_1,\ldots,x_5).$
We want to calculate $J_{\lambda C^T \mathcal{M} C} y$ where $y=[2; 4; -5; 3; 9]$. First we consider 
 $\lambda=1$, our Algorithm 2 becomes Micchelli-Chen-Xu's algorithm. To find the fixed point $u$ of $\mathcal{Q}$, we use $\alpha_k=0.3$ in  Algorithm 2 and stop after $500$ iterations or $\Vert u_{k+1}-u_k\Vert \le 10^{-3}$. Although Micchelli-Chen-Xu's algorithm converges, its output are different with different $\mu$ while they must have the same value if the convergence condition   $ \left\|I - {}{\mu} CC^T\right\| \le 1$ is satisfied.
 \begin{center}
    \begin{tabular} { | c | c | c | c |  c | }
    \hline
     $\mu$ & $10^{-2}$ &$10^{-3}$ \\
    \hline
  $ \left\|I - {}{\mu} CC^T\right\| $ & $4.33$ &$1$\\
  \hline
    $J_{ C^T \mathcal{M} C} y$ & $( 8.73, 15.31, 9.13, 9.45, 22.85)$&$(-0.85, 3.66, -5.01, -1.26, 3.23)$ \\
    \hline
    \end{tabular}
    \end{center}
     \begin{center}
    \begin{tabular} { | c | c | c | c |  c | }
    \hline
     $\mu$ & $10^{-4}$ &$10^{-5}$ \\
    \hline
  $ \left\|I - {}{\mu} CC^T\right\| $ & $1$ &$1$\\
  \hline
    $J_{ C^T \mathcal{M} C} y$ & $( 0.25, 3.69, -5.76, -1.53, 3.34)$&$(0.99, 2.61, -7.20, 0.82, 5.45)$ \\
    \hline
    \end{tabular}
    \end{center}
    Next we use our Algorithm 2 with $\lambda=0.01$. Our algorithm converges and the outputs are the same with different values of $\mu$, even the convergence condition   $ \left\|I - {\lambda}{\mu} CC^T\right\| \le 1$ is not satisfied.
     \begin{center}
    \begin{tabular} { | c | c | c | c |  c | }
    \hline
     $\mu$ & $1$ &$10^{-1}$ \\
    \hline
  $ \left\|I - {}{\lambda\mu} CC^T\right\| $ & $4.33$ &$1$\\
  \hline
    $J_{ \lambda C^T \mathcal{M} C} y$ & $( 1.86, 3.79, -5.27, 2.85, 8.69)$&$( 1.86, 3.79, -5.27, 2.85, 8.69)$ \\
    \hline
    \end{tabular}
    \end{center}
     \begin{center}
    \begin{tabular} { | c | c | c | c |  c | }
    \hline
     $\mu$ & $10^{-2}$ &$10^{-3}$ \\
    \hline
  $ \left\|I - {\lambda}{\mu} CC^T\right\| $ & $1$ &$1$\\
  \hline
    $J_{ \lambda C^T \mathcal{M} C} y$ & $( 1.86, 3.79, -5.27, 2.85, 8.69)$&$( 1.86, 3.79, -5.27, 2.85, 8.69)$ \\
    \hline
    \end{tabular}
    \end{center}
    It means that our Algorithm 2 is not only more general but also more stable than Micchelli-Chen-Xu's algorithm. \qed
\end{example}

Note that $CC^T$ is symmetric and positive semidefinite. If $CC^T$ is positive definite, we even obtain the linear rate convergence of our algorithm.
\begin{theorem}\label{linear}
  If  $\Vert I-{\lambda}{\mu} CC^T \Vert < 1$, then Algorithm 1  in Theorem \ref{algokm} converges with linear rate. Particularly, if  {$E:=CC^T\succ 0$}, we can choose $\lambda>0, \mu >0$ such that the convergence rate of Algorithm 1   is linear.
\end{theorem}
\begin{proof}
If $\Vert I-{\lambda}{\mu} CC^T \Vert < 1$, then $\mathcal{N}$ is a contraction and thus Algorithm 1  in Theorem \ref{algokm} converges with linear rate. 

If $CC^T$ is positive definite there exists $c>0$ such that 
$$
\langle Ex,x \rangle \ge c \Vert x \Vert^2, \;\;\forall\;\;x\in H_1.
$$
{Let  $\gamma:={\lambda}{\mu}$.  We choose $\mu>0$ such that  $\gamma=\frac{c}{\Vert E\Vert^2} \Leftrightarrow \mu=\frac{c}{\lambda \Vert E\Vert^2}$}. Then
\baqn
\Vert (I-\gamma E)x\Vert^2&=&x^2-2\gamma\langle Ex, x \rangle+\gamma^2\Vert Ex\Vert^2.\\
&\le & (1-2\gamma c+\gamma^2\Vert E\Vert^2)\Vert x\Vert^2\\
&\le & (1-\frac{c^2}{ \Vert E\Vert ^2})\Vert x \Vert^2,
\eaqn
where $\Vert E\Vert$ denotes the induced norm of the linear bounded operator $E$.
It means that  $\Vert I-{\lambda}{\mu} CC^T \Vert < 1$, and the conclusion follows. 
\end{proof}
\begin{remark}\normalfont
The condition {$CC^T\succ 0$} holds, for example when $C\in\R^{m\times n}$ is a matrix with full row rank. 
\end{remark}
Next we provide an application of our development to compute equilibria of set-valued Lur'e dynamical systems.
\begin{example}  \label{ex2}
Let us consider a class of set-valued Lur'e dynamical systems of the  following form  
\begin{subequations}
\label{eq:tot}
\begin{empheq}[left={({\mathcal L})}\empheqlbrace]{align}
  & \dot{x}(t) = -f(x(t))+B\lambda(t),\; {\rm a.e.} \; t \in [0,+\infty); \label{1a}\\
  & y(t)=Cx(t),\\
  &  \lambda(t)   \in -\mathcal{M}(y(t)), \;t\ge 0;\\
  & x(0) = x_0.
\end{empheq}
\end{subequations}
where $x: [0,\infty)\to H_1$ is the state variable and $f: H_1\to H_1$ is Lipschitz continuous. The operators $B:H_2\to H_1, C: H_1\to H_2$ are linear bounded and there exists a positive definite linear bounded operator $P$ such that $PB=C^T$ while the set-valued mapping $\mathcal{M}: H_2 \rightrightarrows H_2$ is maximal monotone. Set-valued Lur'e dynamical systems have been a  fundamental model in control theory, engineering and applied mathematics (see, e.g., \cite{ahl2,br0,BT,L1} and  references therein).    Note that  $({\mathcal L})$ can be rewritten as follows
\beq
\dot{x} \in -\mathcal{H}(x), \;\;\;x(t_0) = x_0,
\eeq
where $\mathcal{H}(x)=f(x)+BFCx$. An equilibrium point $x^*$ of $({\mathcal L})$ satisfies 
\beq\label{lure}
0 \in f(x^*)+B\mathcal{M}(Cx^*)\Leftrightarrow 0\in Pf(x^*) + C^T\mathcal{M}(Cx^*).
\eeq
In order to solve (\ref{lure}), it requires to compute the resolvent of the composite operator $C^T\mathcal{M}C$.\qed 
\end{example}

\subsection{Resolvent of $\mathcal{M}_1+C^T\mathcal{M}_2C$}
Next we want to compute the resolvent of $\mathcal{M}_1+C^T\mathcal{M}_2C$ where $\mathcal{M}_1: H_1 \rightrightarrows H_1, \mathcal{M}_2: H_2 \rightrightarrows H_2$ are maximal monotone and $C: H_1\to H_2$ is a linear bounded mapping. For given $\lambda>0$, $y\in H_1$ we want to find $x:=J_{\lambda (\mathcal{M}_1 + C^T \mathcal{M}_2 C)} y\in H_1$ such that 
\beq\label{sum}
y\in x+\lambda\mathcal{M}_1x+\lambda C^T\mathcal{M}_2Cx.
\eeq 

{
\begin{theorem}\label{tmf2}
Let $\lambda > 0$ and $y \in H_1$ be given. Then we have
$$J_{\lambda (\mathcal{M}_1 + C^T \mathcal{M}_2 C)} y=J_{\lambda \mathcal{M}_1}(y - \lambda C^T u),$$
where $u$ is the fixed point of the operator $\mathcal{P}: H_2 \to H_2$ defined by
$$u\mapsto \mathcal{P}(u) := (\mathcal{M}_2)_\kappa \Big(C J_{\lambda \mathcal{M}_1}(y - \lambda C^T u) + \kappa u\Big),$$
 that is for any $\kappa > 0$,
\begin{equation}\label{fixp2}
u = (\mathcal{M}_2)_\kappa \Big(C J_{\lambda \mathcal{M}_1}(y - \lambda C^T u) + \kappa u\Big).
\end{equation}
Furthermore, if $\frac{\lambda}{\kappa} \leq \frac{2}{\|C\|^2}$, then $\|I - \frac{\lambda}{\kappa} CC^T \| \leq 1$, and thus $\mathcal{P}$ is {nonexpansive}.
\end{theorem}
}
\begin{proof}
Note that (\ref{sum}) can be rewritten as follows  
\begin{equation}
\left\{
\begin{array}{l}
y\in  x+\lambda\mathcal{M}_1x+\lambda C^Tu\\ \\
u\in \mathcal{M}_2Cx,
\end{array}\right.
\end{equation}
which is equivalent to
\begin{equation}
\left\{
\begin{array}{l}
 x=J_{\lambda\mathcal{M}_1}(y-\lambda C^Tu)\\ \\
u\in \mathcal{M}_2(CJ_{\lambda\mathcal{M}_1}(y-\lambda C^Tu)). 
\end{array}\right.
\end{equation}
Similarly as in the proof of Theorem \ref{tmf}, we have 
\baqn
u \in \mathcal{M}_2(CJ_{\lambda\mathcal{M}_1}(y-\lambda C^Tu)) &\Leftrightarrow& CJ_{\lambda\mathcal{M}_1}(y-\lambda C^Tu) \in   \mathcal{M}_2^{-1}(u)\\
&\Leftrightarrow& CJ_{\lambda\mathcal{M}_1}(y-\lambda C^Tu)+ \kappa u\in (\mathcal{M}_2^{-1}+\kappa I)u \\
& \Leftrightarrow& u=(\mathcal{M}_2)_\kappa \Big(CJ_{\lambda\mathcal{M}_1}(y-\lambda C^Tu)+ \kappa u\Big).
\eaqn
Let $P_1(u):=CJ_{\lambda\mathcal{M}_1}(y-\lambda C^Tu)$.  We prove that   $P_2(u):=P_1(u)+\kappa u$ is $\kappa$-Lipschitz continuous if $\frac{\lambda}{\kappa}\le 2/\Vert C\Vert^2$  . Indeed, since $J_{\lambda\mathcal{M}_1}$ is firmly-nonexpansive 
we have 
{
\begin{small}
\baqn
 \langle P_1(u_1)-P_1(u_2), u_1-u_2 \rangle 
&=& \langle CJ_{\lambda\mathcal{M}_1}(y-\lambda C^Tu_1)-CJ_{\lambda\mathcal{M}_1}(y-\lambda C^Tu_2), u_1-u_2\rangle\\
&=&-\frac{1}{\lambda} \langle J_{\lambda\mathcal{M}_1}(y-\lambda C^Tu_1)-J_{\lambda\mathcal{M}_1}(y-\lambda C^Tu_2), (y-\lambda C^Tu_1)-(y-\lambda C^Tu_2)\rangle\\
&\le& -\frac{1}{\lambda} \Vert J_{\lambda\mathcal{M}_1}(y-\lambda C^Tu_1)-J_{\lambda\mathcal{M}_1}(y-\lambda C^Tu_2)\Vert^2.
\eaqn
\end{small}
}
Thus if $\frac{\lambda}{\kappa}\le 2/\Vert C\Vert^2$, one has 
{
\baq\nonumber
\Vert P_2(u_1)-P_2(u_2)\Vert^2&=&\kappa^2\Vert u_1-u_2 \Vert^2+2\kappa \langle P_1(u_1)-P_1(u_2), u_1-u_2 \rangle+\Vert P_1(u_1)-P_1(u_2)\Vert^2\\\nonumber
&\le&\kappa^2\Vert u_1-u_2 \Vert^2-(\frac{2\kappa}{\lambda}-\Vert C\Vert^2) \Vert J_{\lambda\mathcal{M}_1}(y-\lambda C^Tu_1)-J_{\lambda\mathcal{M}_1}(y-\lambda C^Tu_2)\Vert^2\\
&\le&\kappa^2\Vert u_1-u_2 \Vert^2.
\label{estsum}
\eaq
}
Consequently, $P_2$ is $\kappa$-Lipschitz continuous and hence $\mathcal{P}$ is {nonexpansive}. 
\end{proof}
{
\begin{theorem}\label{algokm2}
Let $\lambda > 0$ and $y \in H_1$ be given. Choose $\kappa > 0$ such that $\frac{\lambda}{\kappa} \in (0, 2/\|C\|^2)$, and let the sequence $(\alpha_k) \subset (0,1)$ such that $\sum_{k=1}^\infty \alpha_k(1 - \alpha_k) = \infty$. We construct the sequence $(u_k)$ as follows:
\begin{equation*}
\textbf{Algorithm 3:} \quad u_0 \in H, \quad u_{k+1} = (1 - \alpha_k)u_k + \alpha_k\mathcal{P}(u_k), \quad k = 0, 1, 2, \ldots
\end{equation*}
where $\mathcal{P}(u) := (\mathcal{M}_2)_\kappa \Big(C J_{\lambda \mathcal{M}_1}(y - \lambda C^T u) + \kappa u\Big)$.
Then, the sequence $(u_k)$ converges weakly to a fixed point $u$ of $\mathcal{P}$. In addition, if  $\inf \alpha_k > 0$, then $J_{\lambda \mathcal{M}_1}(y - \lambda C^T u_k)\to J_{\lambda (\mathcal{M}_1 + C^T \mathcal{M}_2 C)} (y)$ strongly.
\end{theorem}
}
\begin{proof}
{The weak convergence of $(u_k)$ is easily obtained. For the remain, we do similarly as in the proof of Theorem \ref{algokm}. The sequence $( \Vert u_{k}-u\Vert)$ is decreasing, converges and 
$$ \lim_{k\to \infty}\Vert \mathcal{P}(u_k)-\mathcal{P}(u)\Vert= \lim_{k\to \infty} \Vert u_{k}-u\Vert.$$
Similarly as in (\ref{estsum}), we have 
\baq\nonumber
\Vert \mathcal{P}(u_k)-\mathcal{P}(u)\Vert^2&=&\Vert u_k-u \Vert^2+\frac{2}{\kappa} \langle P_1(u_k)-P_1(u), u_k-u \rangle+\frac{1}{\kappa^2}\Vert P_1(u_k)-P_1(u)\Vert^2\\\nonumber
&\le&\Vert u_k-u \Vert^2-\frac{1}{\kappa^2}(\frac{2\kappa}{\lambda}-\Vert C\Vert^2) \Vert J_{\lambda\mathcal{M}_1}(y-\lambda C^Tu_k)-J_{\lambda\mathcal{M}_1}(y-\lambda C^Tu)\Vert^2\\
&\le&\Vert u_k-u \Vert^2.
\label{estm}
\eaq
Let $k\to\infty$, we must have $ \Vert J_{\lambda\mathcal{M}_1}(y-\lambda C^Tu_k)-J_{\lambda\mathcal{M}_1}(y-\lambda C^Tu)\Vert\to 0$ and the conclusion follows. }
\end{proof}

\section{Conclusions}\label{sec5}
{This paper  introduced a new fixed-point approach for computing the resolvent of composite operators, advancing beyond the classical framework of Micchelli-Chen-Xu \cite{Micchelli,Micchelli1,Moudafi}. The proposed methodology offers several significant theoretical and practical contributions to the field of monotone operator theory and optimization, building upon fundamental work in resolvent operator theory \cite{Bauschke,Robinson}. The primary theoretical contribution lies in the development of a two-parameter fixed-point formulation that generalizes existing single-parameter approaches. We have established that for any $\lambda > 0$ and $y \in H_1$, the resolvent $J_{\lambda C^T\mathcal{M}C}y$ can be expressed as $y - \lambda\mu C^Tu$, where $u$ is the fixed point of a carefully constructed operator $\mathcal{Q}$. This formulation provides enhanced flexibility through the incorporation of both $\lambda$ and $\mu$ parameters, enabling effective computation even when dealing with operators having large norms, a limitation noted in previous works \cite{chen,Moudafi}.}

{Our convergence analysis demonstrates that the proposed algorithms exhibit weak, strong, and linear convergence under verifiable conditions, extending classical results from monotone operator theory \cite{IT,Tseng}. Specifically, when $\|I - \lambda\mu CC^T\| \leq 1$, we prove weak convergence of the iterative sequence, while additional mild conditions on the relaxation parameters ensure strong convergence. Furthermore, we establish that when $CC^T$ is positive definite, appropriate parameter selection yields linear convergence, significantly improving the practical efficiency of the method.}

{The theoretical framework has been extended to address the computation of resolvents for operators of the form $\mathcal{M}_1 + C^T\mathcal{M}_2C$, encompassing a broader class of problems in optimization and control theory \cite{ab,Attouch0}. This extension maintains the convergence properties of the base algorithm while accommodating more complex operator structures encountered in applications such as set-valued Lur'e systems \cite{ahl2,br0,BT}.}

{Several theoretical questions remain open for future investigation. The characterization of optimal parameter selection strategies, particularly the relationship between convergence rates and parameter choices, warrants further study following approaches similar to those in \cite{Attouch1,Chen}. The possibility of weakening the positive definiteness assumption on $CC^T$ in infinite-dimensional settings presents another avenue for theoretical development. Additionally, the connection between our approach and other splitting methods \cite{Fukushima,Tseng} may yield insights into unified frameworks for handling composite operators. From an applications perspective, the development of adaptive parameter selection schemes and the extension to more general classes of structured operators, particularly those arising in hierarchical optimization and multi-leader-follower games \cite{Pang}, represent promising directions for future research. The potential application of our methodology to set-valued Lur'e dynamical systems {\cite{BT,L1} }and traffic equilibrium problems also merits further investigation.}

{In conclusion, this work provides a significant advancement in the computation of resolvent operators, offering both theoretical insights and practical algorithms with provable convergence properties. The framework developed here lays the groundwork for future research in both theoretical and applied aspects of monotone operator theory and optimization.}

\end{document}